\documentclass[leqno,12pt]{article} 
\usepackage{amsmath, amssymb}
\usepackage{amsthm} 
\usepackage{amssymb}
\usepackage{mathrsfs}
\usepackage{amssymb, url, color, graphicx, amscd, mathrsfs}
\usepackage[colorlinks=true, bookmarks=true, pdfstartview=FitH, pagebackref=true, linktocpage=true, linkcolor = magenta, citecolor = blue]{hyperref}
\usepackage{graphicx}
\usepackage{times}
\usepackage{mathtools} 
\usepackage{enumerate} 

\theoremstyle{plain} 
\newtheorem{theorem}{\indent\sc Theorem}[section]
\newtheorem{lemma}[theorem]{\indent\sc Lemma}
\newtheorem{corollary}[theorem]{\indent\sc Corollary}
\newtheorem{proposition}[theorem]{\indent\sc Proposition}

\theoremstyle{definition} 
\newtheorem{definition}[theorem]{\indent\sc Definition}
\newtheorem{remark}[theorem]{\indent\sc Remark}

\numberwithin{equation}{section}
\DeclareMathOperator{\tr}{trace}

\DeclareMathOperator{\Hess}{Hess}
\DeclareMathOperator{\Ric}{Ric}
\DeclareMathOperator{\Rm}{Rm}

\DeclareMathOperator{\I}{Iso}

\makeatletter
\def\address#1#2{\begingroup
\noindent\parbox[t]{7.8cm}{%
\small{\scshape\ignorespaces#1}\par\vskip1ex
\noindent\small{\itshape E-mail address}%
\/: #2\par\vskip4ex}\hfill%
\endgroup}%
\makeatother
\title{On isometry groups of gradient Ricci solitons} 
\author{
\textsc{Ha Tuan Dung and Hung Tran} 
}
\date{} 
\begin{document}

\maketitle

	\begin{abstract}
We give a result estimating the dimension of the Lie algebra of Killing vector fields on an irreducible non-trivial gradient Ricci soliton. Then we study the structure of this manifold when the maximal dimension is attained. There are local and global implications. 
	\end{abstract}
\section{Introduction}\label{s1}
\quad\quad Given an orientable connected $n$-dimensional manifold $M^n$ with a smooth one-parameter family of Riemannian metrics $g(t)$, the Ricci flow is the following
evolution equation
\begin{align}\label{e1.1}
	\frac{\partial g(t)}{\partial t}=-2 {\Ric}(t),
\end{align}
where $\Ric(t)$ is the Ricci curvature of $(M, g(t)).$ The theory was initiated and established in the 80s by Hamilton \cite{Ham1, Ham2, Ham3}. His works  are the premise for G. Perelman \cite{Pe1, Pe2, Pe3}
to eventually resolve the Poincar\'{e} conjecture and Thurston’s geometrization conjecture . Additionally, the Ricci flow is a powerful tool to solve fundamental problems in geometric analysis such as the Differentiable sphere theorem \cite{Br0, Br1}, a version for the curvature of the second kind \cite{CGT}, or the generalized Smale conjecture \cite{Ba1, Ba2}. Regarding the theory's recent developments, we refer the reader to the survey paper of R. Bamler \cite{Ba1} and the references therein.

A gradient Ricci soliton (GRS) is a Riemannian manifold $(M, g)$ that admits a smooth function $f: M \rightarrow \mathbb{R}$ such that
\begin{align}\label{e1.2}
	\Ric+\Hess f=\lambda g
\end{align}
where $\lambda\in\mathbb{R} $  and $\Hess f$ denotes the Hessian of $f$. GRS's correspond to self-similar solutions of Ricci flows and  
arise naturally as singularity models; thus, the study of GRS is of utmost important in the field. There have been extensive investigation on their topology, geometry, and analysis, see \cite{ Br3, C1, C2, CH, Chen, CM, DaW, EMC, Hung, I0, BKo, MS, MW2, Na, Pete1, Pete2}.

A GRS is said to be shrinking, steady, and expanding according to whether $\lambda>0, \lambda=0$, and $\lambda<0,$ respectively. Clearly, an Einstein manifold $N$ can be considered as a
special case of a gradient Ricci soliton where $\Hess f = 0$ and $\lambda$ becomes the Einstein constant. Another basic example is the Gaussian soliton $\left(\mathbb{R}^n, g_{\mathbb{R}^n}, \frac{\lambda|x|^2}{2}\right)$, followed by cylinders $\mathbb{S}^{k} \times \mathbb{R}^{n-k}$ with the product metric where the sphere has Ricci curvature $\lambda.$  Additionally, a combination of those two above is known as \textit{a rank $k$ rigid} gradient Ricci soliton, namely isometric to a  suitable quotient of $N^{k} \times \mathbb{R}^{n-k}$ with $f=\frac{|x|^2}{2}$ on the Euclidean factor \cite{Pe3}. Consequently, a soliton is called \textit{non-trivial} (or \textit{non-rigid}) if at least a factor
in its de Rham decomposition is non-Einstein.

On the other hand, the study of isometric groups plays a pivotal role in classifying the geometric structure of smooth manifolds. Almost a century ago, Dantzig and Waerden \cite{DWa} proved that the group of isometries of a connected, locally compact metric space is locally compact
with respect to the compact-open topology. Later, Myers and Steenrod in the seminal paper \cite{MyS} showed that the group of isometries $\I (M)$ of a Riemannian manifold
$M$ is a Lie transformation group with respect to the compact-open topology. Using the concept of $G$-structure in \cite{MyS}, Kobayashi \cite{Ko3} gave a different proof for the above result of Myers and Steenrod. In particular, Kobayashi proved that there exists a natural
embedding of the isometry group $\I (M)$ into the orthonormal frame bundle $O(M)$ of $M$ such
that $\I (M)$ becomes a closed submanifold of $O (M)$ \cite[page 41]{Ko3}. From this result, he showed that the dimension of
the group of isometries is at most $\frac{1}{2}n(n+1)$ and it is attained if and only if the manifold is of constant curvature \cite{Ko3}. 

While non-gradient Ricci solitons exist in many Lie groups and homogeneous spaces \cite{Bai, Lau}, P. Petersen and W. Wylie \cite{Pete2} remarkably showed that all homogeneous GRS are rigid. Moreover, the authors also proved that if the Riemannian metric is reducible then the
soliton structure is reduced accordingly. Their result is based on the existence of
splitting results induced by Killing vector fields. Recently, the second author \cite{Hung} has demonstrated some results about the group of isometries of an irreducible non-trivial K\"{a}hler GRS.

Inspired by their work, we will study the isometry group $\I (M)$ and its Lie algebra of an irreducible non-trivial gradient Ricci soliton $(M, g, f).$  Indeed, we will determine the maximal dimension and characterize the case when it's attained. Towards that end, we recall
\[\mathfrak{iso}(M, g):= \{X \text{ is a smooth tangent vector field on $M$}, \mathcal{L}_X g=0\}.\]
Closely related is the Lie algebra of Killing vector fields preserving $f$:
\[ \mathfrak{iso}_f(M, g, f):= \{X \text{ is a smooth tangent vector field on $M$}, \mathcal{L}_X g=0=\mathcal{L}_X f\}.\]
Our first result investigates the maximal dimension of $\mathfrak{iso}_f(M, g, f)$.
 \begin{theorem}\label{main1}
		Let $(M^n, g, f),$ with $n\geq 3,$ be a GRS. If $f$ is non-constant then $\mathfrak{iso}_f(M, g, f)$ is of dimension at
		most $\frac{1}{2}\left(n-1\right)n$ and equality happens iff each connected component of a regular level set of $f$ is a space of constant curvature. 
		
		Let $(\mathbb{N}^{n-1}, g_{\mathbb{N}})$ denote the space form model. If $g_\mathbb{N}$ is non-flat, the equality happens iff the metric is locally a warped product. That is, there is an open dense subset such that around each of its points, there is a neighborhood diffeomorphic to a product $I \times \mathbb{N}$ and the metric $g$ is given by
		\begin{align}\label{1.1}
			g=d t^2+F^2(t) g_{\mathbb{N}}.
		\end{align}
	Here, $I$ is an open interval and $F: I \mapsto \mathbb{R}^{+}$ is a smooth function. 
\end{theorem}

Furthermore, it is possible to relax the assumption on preserving $f$. A Riemannian manifold is locally irreducible if it is not a local Riemannian product metric around each point.
\begin{theorem}\label{main2}
Let $(M^n, g, f),$ with $n\geq 3,$ be a locally irreducible non-trivial GRS. Then $\mathfrak{iso}(M, g)$ is of dimension at most $\frac{1}{2}\left(n-1\right)n$. In addition, equality happens iff it is smoothly constructed as in the case of equality of Theorem \ref{main1}.
\end{theorem}
\begin{remark}
	The above theorems are essentially local. That is, there is no mentioning of the completeness of the metric.
\end{remark} 
Indeed, the soliton structure is so rigid that it is difficult to complete the above metrics.
\begin{corollary} \label{main3}
	Let $(M^n, g, f),$ with $n\geq 3,$ be an irreducible non-trivial complete GRS. Then $\mathfrak{iso}(M, g)$ is of dimension at most $\frac{1}{2}\left(n-1\right)n$. For $\lambda\geq 0$, equality happens iff $\lambda=0$ and it is isometric to a Bryant soliton.

\end{corollary}
Finally, we observe that there is a gap in the dimension. 
\begin{corollary} \label{main4} 
	Let $(M^n, g, f),$ with $n\neq 5,$ be an irreducible non-trivial GRS and let $d:=\dim \mathfrak{iso}(M, g)$. If $d<\frac{1}{2}\left(n-1\right)n$ then $d\leq \frac{1}{2}(n-1)(n-2)+1$.  	
	
\end{corollary}
\begin{remark}
	The completeness of these metrics will be addressed elsewhere. 
\end{remark}

The paper is organized as follows. In Section \ref{s2}, we recall basic notations and collect preliminary materials that we will use in the rest of the paper. The main results, Theorem \ref{main1}, Theorem \ref{main2}, and Corollary \ref{main3},  will be proved in Secion \ref{s3}. Finally the Appendix considers the case that each level set of a GRS is Euclidean. 

\section{Preliminaries} \label{s2}
\quad This section is to recall auxiliary results on Killing vector fields, group actions on manifolds, and gradient Ricci solitons. The main references are \cite{BA, Cho, Ko1, Ko2, Ko3, Pete2,  Pete3}.  
\subsection{Killing vector fields and group actions on manifolds}
\quad 
In this subsection, we briefly review basic properties of Killing vector fields and their relationship to the isometry group.  Besides, we also recall some basic concepts related to group actions on manifolds. The standard texts are \cite{BA, Ko1, Pete3}. We begin by providing the definition of Riemannian isometries.
\begin{definition}
 Let $\left(M, g_M\right)$ and $\left(N, g_N\right)$ be Riemannian manifolds. An isometry from $M$ to $N$ is a diffeomorphism $\phi: M \rightarrow N$ such that $$\phi^*\left(g_N\right)=g_M.$$ In
 other words, $\phi$ is an isometry if for all $p \in M$ and tangent vectors $X_p, Y_p \in T_p M$,  
 $$
 \left.g_M\right|_p\left(X_p, Y_p\right)=\left.g_N\right|_{\phi(p)}\left(\phi_* X_p, \phi_* Y_p\right).
 $$
In this sense, we say that $\phi$ preserves the metric structure. In addition, $M$ and $N$ are called isometric.
\end{definition}
	The set of all isometries of a Riemannian manifold $(M, g)$ onto itself forms a group (indeed a Lie group), which is denoted by $\I(M)$ and called the isometry group of $M.$
\begin{definition}
 A vector field $X$ on a Riemannian manifold $(M, g)$ is called a Killing vector field if the Lie derivative with respect to $X$ of the metric $g$ vanishes, i.e.,
 $\mathcal{L}_X g=0.
 $
\end{definition}
The following proposition shows the relationship between Killing vector fields and isometries. For a proof, we refer the reader to \cite[Proposition 8.1.1]{Pete3}.
\begin{proposition}\label{p1}
A vector field $X$ on a Riemannian manifold $(M, g)$ is a Killing vector if and only if the local
 flows generated by $X$ act by isometries. 
\end{proposition}

	Because of Proposition \ref{p1}, Killing vector fields are also commonly known as \textit{infinitesimal isometries}, a terminology that arises from the idea of integrating vector fields to obtain isometries. Furthermore, they enjoy strong analytic properties.

\begin{proposition}{\rm \cite[Proposition 8.1.4]{Pete3}}
 Let $X$ be a Killing vector field on a Riemannian manifold $(M, g)$. If there exists a point $p \in M$ such that
 	$X_p=0$ and $(\nabla X)_p=0$,
 then $X$ is identical $0.$
\end{proposition}

\begin{remark}
 The set of all Killing vector fields on a Riemannian manifold $(M, g)$ is a Lie algebra, and denoted by $\mathfrak{iso}(M, g).$ Furthermore, by Theorem 8.1.6 in \cite{Pete3}, if the Levi-Civita connection induced by the Riemannian metric $g$ on $M$ is complete, then so is each Killing vector field. In that case, $\mathfrak{iso}(M, g)$ is the Lie algebra of $\operatorname{Iso}(M)$.
\end{remark}
Next, we recall a result estimating the dimension of the Lie algebra $\mathfrak{iso}(M, g)$ and $\I (M, g)$, which will play an important role in our proof of Theorem \ref{main1}.
\begin{theorem}{\rm \cite[Theorem 8.1.6]{Pete3} \cite[Theorem 1, Note 10]{Ko1}}
	\label{t1}
	Let $(M,g)$ be a connected Riemannian manifold of dimension $n.$ 
	Then the Lie algebra $\mathfrak{iso}(M, g)$ is of dimension at most $\frac{1}{2} n(n+1).$ If $\dim \mathfrak{iso}(M, g)=\frac{1}{2} n(n+1)$, then $M$ is a space of constant curvature. Furthermore, if $\dim {\I}(M)=\frac{1}{2} n(n+1)$, then $M$ is isometric to one of the following:
	\begin{enumerate} [\rm (i)]
		\item an $n$-dimensional Euclidean space $\mathbb{R}^n$,
		\item an $n$-dimensional sphere $\mathbb{S}^n$,
		\item an $n$-dimensional real projective space,
		\item an $n$-dimensional, simply connected hyperbolic space. 
	\end{enumerate}
\end{theorem}

In the rest of this subsection, we recall some basic notions about group actions on manifolds following the book by Alexandrino and Bettiol \cite{BA}.
\begin{definition} 
 Let $G$ be a Lie group and $M$ a smooth manifold. A smooth map $\l: G \times M \rightarrow M$ is called \textit{a (left) action} of $G$ on $M$, or \textit{a (left) $G$-action} on $M$, if
 \begin{itemize}
 	\item[(i)] $\l(e, x)=x$, for all $x \in M,$ where $e$ is the identity element of $G;$
 	\item[(ii)] $\l\left(\mathsf{g}_1, \l\left(\mathsf{g}_2, x\right)\right)=\l\left(\mathsf{g}_1 \mathsf{g}_2, x\right)$, for all $\mathsf{g}_1, \mathsf{g}_2 \in G$ and $x \in M$.
 \end{itemize}

\end{definition}
We often write $\mathsf{g} \cdot x$ or just $\mathsf{g} x$ in place of the more pedantic notation $l(\mathsf{g}, x)$. A right action $r: M \times G \rightarrow M$ can be defined analogously and we write $x \cdot \mathsf{g}$ or $x \mathsf{g}$.
\begin{definition}
An action is said to be \textit{proper} if the associated map
$G \times M \mapsto M \times M$, given by \begin{align}\label{2.1}
G \times M \ni(\mathsf{g}, x) \longmapsto(\mathsf{g}\cdot x, x) \in M \times M
\end{align}
is \textit{proper}, i.e., if the preimage of any compact subset of $M \times M$ under \eqref{2.1} is a compact subset of $G \times M$.
\end{definition}

From Proposition 3.62 and Theorem 3.65 in \cite{BA}, we see that  actions by closed subgroups of isometries are
proper, and conversely every proper action can be
made isometric with respect to a certain Riemannian metric.

\begin{definition}\label{d2.14}
A Riemannian
manifold $(M, g)$ is said to be \textit{homogeneous} if its isometry group acts \textit{transitively}, i.e., for each pair of points $x, y \in M$ there is a $\mathsf{g} \in \I (M)$ such that $ \mathsf{g}\cdot x=y.$
\end{definition}
\subsection{Gradient Ricci solitons}\label{s2.3}
\quad In this subsection, we shall recall some basic facts and collect preliminaries about GRS. Let $\left(M, g, f\right)$ be a GRS of dimension $n \geq 3.$ The Ricci curvature $\Ric,$ the scalar curvature $S$ and the poential function $f$ are related by the following equations \cite[Proposition 2.1]{EMC}.
\begin{proposition}\label{2.15}
For any gradient Ricci soliton $\left(M, g, f\right)$, we have
\begin{align}\label{p2.3}
\Delta_f S+2|\Ric|^2=2 \lambda S,
\end{align}
\begin{align}\label{p2.15}
S+|\nabla f|^2-2 \lambda f=C
\end{align}
for some constant $C.$ Here $\Delta_f$ denotes the $f$-Laplacian, $\Delta_f=\Delta-g\left(\nabla f, \cdot\right).$
\end{proposition}
In \cite[Proposition 2.1]{Pete2}, Petersen and Wylie proved the following result about a Killing fields on a GRS.
\begin{proposition}\label{p2}
If $X$ is a Killing field on a gradient Ricci soliton $(M, g, f),$ then $\nabla (Xf)$ is parallel. Moreover, if $\lambda \neq 0$ and $\nabla (Xf)=0$, then also $X f=0$. 
\end{proposition}
\begin{remark}
We emphasize that to prove the above result, Petersen and Wylie used the condition that scalar curvature is bounded. However, such assumption can be omitted since, for $\lambda\neq 0$, the scalar curvature of a GRS is always bounded from either below or above \cite[Theorem 8.6]{AM}). This is enough for the argument to go through.
\end{remark}
Consequently, Petersen and Wylie \cite{Pete2} gave the following splitting result.
\begin{lemma}{\rm \cite[Corollary 2.2]{Pete2}} \label{l2.17}
  If $X$ is a Killing field on a GRS $\left(M, g, f, \lambda\right)$ then either $\nabla(Xf)=0$ or $M$ locally splits a line isometrically. The latter means that, around each point $p$, there is a neighborhood $U=V \times I$, where $V$ is an open neighborhood of a submanifold and $T_pV\perp (\nabla (Xf))_p$ and $I$ is an open interval. The Riemannian metric in $U$ is the direct product of the induced metrics on each factor and $(V, g_{\mid V}, f)$ is a GRS.
\end{lemma}

For $\left(M, g, f\right)$ a shrinking gradient Ricci soliton,
upon scaling the metric $g$ by a constant, we can assume that $\lambda=\frac{1}{2}$.
Then the equation \eqref{e1.2} takes the form
\begin{align}\label{e3.9}
	\Ric+\nabla^2 f=\frac{1}{2} g.
\end{align}
By adding a constant to $f$ if necessary and the equation \eqref{p2.15}, we may normalize the soliton such that
\begin{align}\label{e3.10}
	S+|\nabla f|^2=f.
\end{align}
Moreover, according to a result by Chen \cite[Corollary 2.5]{Chen} (see also \cite[Theorem 8.6]{AM}), we have $S \geq 0$ for any shrinking gradient Ricci soliton. This and \eqref{e3.9} entail that $f\geq 0.$ On the other hand,  from Haslhofer-M\"{u}ller's works \cite[Lemma 2.1]{HR}
 (see also \cite[Theorem 1.1]{C2}), we know that the potential function $f$ has quadratic growth at infinity. Using these results, we obtain the following proposition.
\begin{proposition}\label{po2.15}
Let $(M, g, f)$ be an $n$-dimensional complete noncompact shrinking gradient Ricci soliton with \eqref{e3.9} and \eqref{e3.10}. Then each regular level set of $f$ is a compact set.
\end{proposition}
\begin{proof}[Proof of Proposition \ref{po2.15}]
For each regular value $c\in f\left(M\right),$ we consider the level set $M_c$ of $f$.
Since $f$ is a smooth function and $\{c\}$ is a closed set, $M_c=f^{-1}(c)$ is also a closed set. By Lemma 2.1 in \cite{HR}, there exists a point $p \in M$ where $f$ attains its infimum and $f$ satisfies the following quadratic growth estimate
\begin{align}
	\frac{1}{4}\left[\left(r(x)-5 n\right)_{+}\right]^2 \leq f(x) \leq \frac{1}{4}\left(r(x)+\sqrt{2 n}\right)^2, \nonumber
\end{align}
where $r(x)$ is a distance function from $p$ to $x$, and $a_{+}=\max \{a, 0\}$ for $a \in \mathbb{R}$. This and the fact that $f\geq 0$ imply that $M_c$ is a bounded set and therefore $M_c$ is a compact set.
\end{proof}

\section{Dimension bound and Rigidity} \label{s3}
\quad This section is devoted to the proof of our main results. Let $\left(M, g, f\right)$ be a GRS  of dimension $n \geq 3.$ 
Recall
\[\mathfrak{iso}(M, g):= \{X \text{ is a smooth tangent vector field on $M$}, \mathcal{L}_X g=0\}.\]
We also define 
\begin{align}\label{e3.2}
\mathfrak{iso}_f(M, g, f):= \{X \text{ is a smooth tangent vector field on $M$}, \mathcal{L}_X g=0=\mathcal{L}_X f\}.
\end{align}
Then, we see that $\mathfrak{iso}_f(M, g, f)\subset \mathfrak{i s o}\left(M, g\right)$ is a vector subspace. Towards our goal, we will establish the following lemma concerning $\mathfrak{i s o}\left(M, g\right)$.
\begin{lemma}\label{l3.1}
If $X\in \mathfrak{iso}(M, g)$ and $g(X,\nabla f)=Xf$ is constant then $[X, \nabla f]=0$. 
\end{lemma}
\begin{proof}
We observe that
	\begin{align} \label{3.1}
		g\left(\mathcal{L}_X \nabla f, Y\right) & =g\left(\nabla_X \nabla f-\nabla_{\nabla f} X, Y\right) \nonumber\\
		& =(\operatorname{Hess} f)(X, Y)+g\left(\nabla_Y X, \nabla f\right)-\left(\mathcal{L}_X g\right)(Y, \nabla f) \nonumber\\
		& =(\operatorname{Hess} f)(X, Y)-g\left(X, \nabla_Y \nabla f\right)+Y\left(\mathcal{L}_X f\right)-\left(\mathcal{L}_X g\right)(Y, \nabla f) \nonumber\\
		& =Y\left(\mathcal{L}_X f\right)-\left(\mathcal{L}_X g\right)(Y, \nabla f)
	\end{align}
	for any $Y\in TM.$ Since $X$ is a Killing vector field, $\left(\mathcal{L}_X g\right)(Y, \nabla f)=0.$ Since $\mathcal{L}_X f=Xf$ is a constant, $Y\left(\mathcal{L}_X f\right)=0.$ Combining these results yields $[X, \nabla f]=\mathcal{L}_X \nabla f=0.$\end{proof}

We now give the proof of Theorem \ref{main1}.
\begin{proof}[Proof of Theorem \ref{main1}]
 Let $M_c$ be a level set of $f$ with the induced metric $g_c:=g_{\mid TM_c}$, where $c\in f\left(M\right)$ is a regular value. By the level set theorem \cite{Tu}, $(M_c, g_c)$ is a smooth submanifold of co-dimension one. Consider $X\in \mathfrak{iso}_f(M, g, f)$ and let $\varphi_t^X$ denote the local flow generated by the vector field $X.$ Then, we have 
$
X=\left.\frac{d}{d t}\right|_{t=0} \varphi_t^X.
$
Since $\mathcal{L}_X g=0$ and $\mathcal{L}_X f=0,$ we deduce that 
$\left(\varphi_t^X\right)^* g=g$ and 
\begin{align}\label{e3.1}
\left(\varphi_t^X\right)^* f=f \Leftrightarrow f \circ \varphi_t^X=f,
\end{align}
where $\left(\varphi_t^X\right)^*$ is the pull-back of $\varphi_t^X.$  By Proposition \ref{p1}, we see that $\varphi_t^X: M \rightarrow M$ generates local isometries 
	and $\varphi_t^X\left(M_c\right) \subseteq M_c.$ From this, we notice that $\varphi_t^X$ induces a map $\left.\widetilde{\varphi}_t^X \equiv \varphi_t^X\right|_{M_c}: M_c \rightarrow M_c$. We consider the vector field
$$
\widetilde{X}=\left.X\right|_{M_c}=\left.\frac{d}{d t}\right|_{t=0}\left(\left.\varphi_t^X\right|_{M_c}\right)=\left.\frac{d}{d t}\right|_{t=0} \widetilde{\varphi}_t^X.
$$
Since $\widetilde{\varphi}_t^X$ is an isometry on $M_c,$ we conclude that $\widetilde{X} \in \mathfrak{i s o}\left(M_c, g_c\right)=\left\{X \in T M_c \mid \mathcal{L}_X g_c=0\right\}.$ Thus, the map 
$$
\begin{aligned}
	\pi: \mathfrak{i s o}_f(M, g, f) & \rightarrow \mathfrak{i s o}\left(M_c, g_c\right) \\
	X & \mapsto \pi(X):=\left.\widetilde{X} = X\right|_{M_c}
\end{aligned}
$$
is well defined. Moreover, $\pi$ is a linear map. Next, we will prove that $\pi$ is injective. Suppose that $\left.X\right|_{M_c}\equiv 0,$ where $X\in \mathfrak{iso}_f(M, g, f).$ Since $\mathcal{L}_X g=0$ and $\mathcal{L}_X f=0,$ Lemma \ref{l3.1} yields 
\begin{align}\label{3.2}
	[X, \nabla f]=\mathcal{L}_X \nabla f=0.
\end{align}
 Let $p\in M_c$ and $Y\in TM$ be an arbitrary vector field. Then, we have $Y=Z+W,$ where $Z\in TM_c$ and $W \in T^{\perp} M_c.$ Since $\nabla f$ is a normal vector field of $TM_c,$ $W=\eta\nabla f,$ where $\eta$ is a smooth function. Therefore, we get
	\begin{align}\label{3.3}
		\left.\left(\nabla_Y X\right)\right|_p &=\left.\left(\nabla_{Z+W} X\right)\right|_p
		=\left.\left(\nabla_Z X\right)\right|_p+\left.\eta\left(\nabla_{\nabla f} X\right)\right|_p=\left.\eta\left(\nabla_{\nabla f} X\right)\right|_p.
	\end{align}
	The last equality follows from $\left.X\right|_{M_c}=0$. Furthermore, using \eqref{3.2}, we compute
	\begin{align}\label{3.4}
		\left.\left(\nabla_{\nabla f} X\right)\right|_p=\left.\left(-[X, \nabla f]+\nabla_X \nabla f\right)\right|_p=(\nabla_X \nabla f)|_p=0 .
	\end{align}
Since $Y\in TM$ is an arbitrary vector field, we conclude that $\left.(\nabla X)\right|_p=0.$ Since $X_p=0$, by Proposition \ref{p1}, we deduce that $X\equiv 0.$ This shows that the map $\pi$ is injective. From Theorem \ref{t1} and note that $\dim M_c=n-1,$  we obtain 
\begin{align}\label{3.6}
\dim\mathfrak{i s o}_f(M, g, f) \leq \dim\mathfrak{i s o}\left(M_c, g_c\right) \leq \frac{1}{2}(n-1) n.
\end{align}

Next, we will consider the case $\dim\mathfrak{i s o}_f(M, g, f)=\frac{1}{2}(n-1) n.$ By Theorem \ref{t1}, each regular connected component of $f$ with the induced metric must be of constant curvature. Consequently, each is homogeneous and complete \cite[Theorem IV.4.5]{Ko1}. Thus, the Lie algebra $\mathfrak{i s o}_f(M, g, f)$ indeed generates a global group of isometries on $(M, g)$ and the action is transitive on each regular level set. Therefore $S$ is contant on each regular level set and, by Prop \ref{2.15}, so is $|\nabla f|$ and $\frac{df}{|\nabla f|}$ is closed and locally exact. Define $t$ by $dt=\frac{df}{|\nabla f|}$ then the metric can be written locally as 
\[ g= dt^2+g_t,\] 
where $g_t$ is a family of metrics on the differentiable manifold corresponding to a regular connected component. Let $L$ denote the shape operator and \[\nu := \frac{\partial g_t}{\partial t}=  2 g_t\circ L. \]
Furthermore, by the constancy of $|\nabla f|$ on each regular connected component, singular values for $f: M \mapsto \mathbb{R}$ are isolated. By continuity, nearby connected components must be obtained from the same model space $(\mathbb{N}^{n-1}, g_{\mathbb{N}})$.

Since $g_t$ is homogeneous, so is $\nu$ and it suffices to consider its value at a point. We recall the evolution of the Ricci tensor, $\Ric_t:= \Ric(g_t)$, \cite[page 109]{Cho}, for normal coordinates, 
\begin{align*}
	\frac{\partial}{\partial t} \Ric_{ij} &= -\frac{1}{2}\Big(\Delta_L \nu_{ij}+\nabla_i\nabla_j \tr(\nu))-\nabla_i (\delta \nu)_j-\nabla_j(\delta v)_i \Big),\\
	\Delta_L \nu_{ij} &= \Delta \nu_{ij}+2\Rm_{kijl}\nu_{kl}-\Ric_{ik}\nu_{jk}-\Ric_{jk}\nu_{ik},\\
	\widehat{\Rm}\,(\nu)_{ij} &=  2\Rm_{kijl}\nu_{kl}-\Ric_{ik}\nu_{jk}-\Ric_{jk}\nu_{ik}. 
\end{align*} 
As $\nu$ is homogeneous, all spacial derivatives vanish.  

\begin{flushleft}
	\textbf{Claim}. If $g_{\mathbb{N}}$ is non-flat then $\nu$ is a multiple of $g_{\mathbb{N}}$.
\end{flushleft}

\textit{Proof of the claim}. Since $g_t$ is isomorphic to a space form, $\Ric$ is a multiple of the metric. Thus, $\Ric$, when considered as a linear map on the tangent space, is a multiple of the identity for each $t$. Thus, so is its derivative. If $\Rm(g_{\mathbb{N}})\neq 0$ then $\widehat{\Rm}(\nu)$ is a linear combination of a non-trivial multiple of $\nu$ and a multiple of the identity. The result then follows. 

 Thus, if $\Rm(g_{\mathbb{N}})\neq 0$  there is a local diffeomorphism $\phi: I \times \mathbb{N} \mapsto U,$ an open neighborhood in $M$, such that 
$$
\phi^*(g)=\phi^\ast(d t^2+g_t)=d t^2+F^2(t) \pi^* g_{\mathbb{N}}. 
$$    
The result then follows. \end{proof}
\begin{remark}
	The case that $g_{\mathbb{N}}$ is flat means each level set is an Euclidean space. Their analysis will be carried out in the Appendix. 
\end{remark}

Next, we will apply Theorem \ref{main1} to prove Theorem \ref{main2}.
\begin{proof}[Proof of Theorem \ref{main2}]  Since $M$ is locally irreducible, by Lemma \ref{l2.17}, $\nabla (Xf)\equiv 0$ for any Killing vector field $X\in \mathfrak{iso}(M, g)$. We then consider two possible cases.
	
\textbf{Case 1}: $\lambda \neq 0.$ By Prop. \ref{p2}, $Xf=0$. That is, each Killing vector field automatically preserves $f$. Thus, $\mathfrak{iso}(M, g)\equiv \mathfrak{iso}_f(M, g, f)$ and the result then follows from Theorem \ref{main1}. 

\textbf{Case 2}: $\lambda=0.$ If the scalar curvature $S$ of $(M, g, f)$ is a constant then from \eqref{p2.3}, we obtain $\Ric \equiv 0,$ and hence $(M, g, f)$ is Ricci-flat, which is a contradiction to our non-triviality assumption. Thus, $S$ is non-constant and one observes that it is is invariant under isometries. Hence
\begin{align}
	\mathfrak{iso}_S(M, g, f):= \{X \text{ is a smooth tangent vector field on $M$}, \mathcal{L}_X g=0=\mathcal{L}_X S\}= \mathfrak{iso}(M, g). \nonumber
\end{align}
Repeating the argument as in the proof of Theorem \ref{main1} we have, for $M_c$ a regular level set of $S$,  
\begin{align}
\dim \mathfrak{iso}(M, g)=\dim\mathfrak{i s o}_S(M, g, f) \leq \dim\mathfrak{i s o}\left(M_c, g_c\right) \leq \frac{1}{2}(n-1) n. \nonumber
\end{align}
If the equality happens then, by Theorem \ref{t1}, each regular connected component of $S$ with the induced metric must be of constant curvature. Furthermore, $|\nabla S|$ is also invariant by the isometric action and the rest is verbatim as in the proof of Theorem \ref{main1}. \end{proof}

\begin{proof}[Proof of Corollary \ref{main3}]
 First, by Theorem \ref{main2}, $\dim\mathfrak{iso}(M, g)\leq \frac{1}{2}(n-1)n$  and equality happens only if each connected component of a regular level set of $f$ is a space form $(\mathbb{N}^{n-1}, g_{\mathbb{N}}).$  We now suppose that $\dim\mathfrak{iso}(M, g)= \frac{1}{2}(n-1)n$. Consequently, each regular level set is homogeneous and complete and, consequently, $\mathfrak{iso}(M, g)$ is the Lie algebra of the isometry group on each regular connected component. We'll divide the rest of the proof into cases.
 
\textbf{Case 1}: $\lambda>0.$ By Proposition \ref{po2.15} each regular connected component is compact. Then, by Theorem \ref{t1}, the model space $(\mathbb{N}, g_{\mathbb{N}})$ must be spherical (round sphere or the real projective space). Then, from Theorem \ref{main2}, the Riemannian metric is a local warped product
$$
g= dt^2+ F^2(t)g_{\mathbb{N}}.$$
 By \cite[Theorem 1]{Vas}, the Weyl tensor is vanishing and $(M, g)$ is locally conformally flat. The classification of a GRS with such property for $\lambda>0$ is well-known.  By \cite[Theorem 1]{BKo} and \cite[Theorem 2]{Zh} (or see also \cite[Proposition 2.4]{CaM}), $(M, g, f)$ must be either the Gaussian shrinking gradient Ricci soliton on $\mathbb{R}^n$, the round cylinder shrinker on $\mathbb{S}^{n-1} \times \mathbb{R}$, or the round sphere shrinker on $\mathbb{S}^n.$ They are all rigid. 
	
\textbf{Case 2}: $\lambda=0.$ If the metric $g_{\mathbb{N}}$ is flat, that is $(M^n, g, f)$ $(n \geq 3)$ is a steady gradient Ricci solition with local Euclidean level sets then by Proposition \ref{p4.2} in the Appendix, the metric $g$ is incomplete. This shows that the metric $g_{\mathbb{N}}$ is non-flat. Then by Theorem \ref{main2} and \cite[Theorem 1]{Vas}, $(M, g)$ is locally conformally flat. According to \cite[Theorem 2]{C3}, $(M, g, f)$ is either the Gaussian soliton or isometric to the Bryant soliton. \end{proof}

\begin{proof}[Proof of Corollary \ref{main4}] The proof is by contradiction. Suppose that $d>\frac{1}{2}(n-1)(n-2)+1$. The group of isometries on $(M, g)$ generates a Lie algebra of complete Killing vector fields which is a sub-algebra of $\mathfrak{iso}(M, g)$. From the proofs of Theorem \ref{main1} and \ref{main2}, there is an injective map from $\mathfrak{iso}(M, g)$ to that of a co-dimension one regular submanifold $(M_c, g_c)$. Furthermore, the completeness of a vector field is preserved under the map. Thus, for each regular connected component, \[\text{dim}(\I (M_c, g_c))> \frac{1}{2}(n-1)(n-2)+1.\]
By \cite[Theorem 3.2]{Ko3}, $\text{dim}(\I (M_c, g_c))=\frac{1}{2}n(n-1)$ and each $(M_c, g_c)$ is a space form which is homogeneous and complete. Thus, by continuity, we go back to the case of Theorem \ref{main2} and $d= \frac{1}{2}n(n-1)$, a contradiction. 
\end{proof}

\section{Appendix}

\quad In this Appendix we consider the case of each level set of a GRS is Euclidean, which was mentioned in the proof of Theorem \ref{main1}. We first adapt the gradient Ricci soliton equation \eqref{e1.2} to the cohomogeneity one setting, essentially using the methodology and notation of \cite{DaW}.

Let $G$ be a Lie group acting isometrically on
a Riemannian manifold $(M, g).$ The action is of
\textit{cohomogeneity one} if the orbit space $M/G$ is one-dimensional. In this case, we choose a unit
speed geodesic $\gamma(t)$ that intersects all principal orbits perpendicularly. Then, it is possible
to define a $G$-equivariant diffeomorphism
$
\Phi: I \times P \mapsto M_0$
given by $\Phi(t, h K)=h \cdot \gamma(t) .
$
Here, $M_0 \subset M$ is an open dense subset, $I$ is an open interval;  $P=G / K$ where $K$ is the istropy group along $\gamma(t)$. Then the pullback metric  is
of the form
$$
\Phi^*(g)=d t^2+g_t
$$
where $g_t$ is a one-parameter family of $G$-invariant metrics on $P$. We let $L$ denote the shape operator
$L(X)=\nabla_X N,$
where $N=\Phi_*\left({\partial_t}\right)$ is a unit normal vector field. We will consider $L_t=L_{\mid \Phi(t \times P)}$ to be a one-parameter family of endormorphisms on $T P$ via identification $T(\Phi(t \times P))=T P$. Following \cite{DaW}, we have
$
\partial_t g=2 g_t \circ L_t,
$
that is for $X, Y \in T P$,
$$
\left(\partial_t g\right)(X, Y)=2 g_t\left(L_t(X), Y\right).
$$
From Gauss, Codazzi, and Riccati equations, we find that the Ricci curvature of $(M_0, g)$ is totally determined by the geometry of the shape operator and how it evolves. Moreover, if the function $f$ is invariant by the group action, then
the gradient Ricci soliton equation \eqref{e1.2} is reduced to
\begin{align}\label{e2.4}
	0 & =-(\delta L)-\nabla \operatorname{tr} L, \nonumber\\
	\lambda & =-\tr\left(L^{\prime}\right)-\tr\left(L^2\right)+f^{\prime \prime}, \\
	\lambda g_t(X, Y) & =\Ric_t(X, Y)-(\tr L) g_t(L (X), Y)-g_t\left(L^{\prime}(X), Y\right)+f^{\prime} g_t(L (X), Y).\nonumber
\end{align}
where $\Ric_t$ denotes the Ricci curvature of $\left(P, g_t\right), \delta L=\sum_i \nabla_{e_i} L\left(e_i\right)$ for an orthonormal basis and $\operatorname{tr} T=\operatorname{tr}_{g_t} T_t$.

Now, we consider a GRS $(M^n, g, f)$ $(n \geq 3)$  with local Euclidean level sets:
\begin{align}\label{e4.2}
g:=d t^2+g_t=d t^2+\sum_i h_i^2(t) d x_i^2,
\end{align}
where each function $h_i$ is smooth. Observe that
$$
2 g_t {\circ} L_t=\partial_t g_t=2 \sum_i \frac{h_i^{\prime}}{h_i} h_i^2 d x_i^2.
$$
From this, we get 
$$
L_t\left(\partial_i\right)=\frac{h_i^{\prime}}{h_i} \partial_i, \quad L_t^2\left(\partial_i\right)=\left(\frac{h_i^{\prime}}{h_i}\right)^2 \partial_i,
$$
and
$$
L_i^{\prime}\left(\partial_i\right)=\left(\frac{h_i^{\prime \prime}}{h_i}-\left(\frac{h_i^{\prime}}{h_i}\right)^2\right) \partial_i.
$$
Consequently, 
\begin{align*}
	\tr{L_t} &=\sum_i \frac{h_i'}{h_i},\\
	\tr{L^2_t} &=\sum_i \left(\frac{h_i'}{h_i}\right)^2,\\
	\tr{L'_t} &=\sum_i\left(\frac{h_i^{\prime \prime}}{h_i}-\left(\frac{h_i^{\prime}}{h_i}\right)^2\right).
\end{align*}
Since the shape operator $L$  satisfies the Riccati equation \cite[page 117]{EW}, the sectional curvature of the $2$-plane section spanned by $e_i=\frac{\partial_i}{h_i}$ and $N$ is given by
$$
\begin{aligned}
	K\left(e_i, N\right) & =g\, {\circ}\left(-L^{\prime}-L^2\right)\left(e_i, e_i\right) \\
	& =-g\left(L^{\prime}\left(e_i\right), e_i\right)-g\left(L^2\left(e_i\right), e_i\right) \\
	& =-g\left(\left(\frac{h_i^{\prime \prime}}{h_i}-\left(\frac{h_i^{\prime}}{h_i}\right)^2\right) e_i, e_i\right)-g\left(\left(\frac{h_i^{\prime}}{h_i}\right)^2 e_i, e_i\right)=-\frac{h_i^{\prime \prime}}{h_i}.
\end{aligned}
$$
Using the Gauss equation \cite[Theorem 3.2.4]{Pete3}, we see that  the sectional curvature of the $2$-plane section spanned by $e_i$ and $e_j$ is given by
$$
K\left(e_i, e_j\right)=-g\left(L\left(e_i\right), e_i\right) g\left(L\left(e_j\right), e_j\right)=-\frac{h_i^{\prime}}{h_i} \frac{h_j^{\prime}}{h_j} .
$$
The Ricci curvature is then given by 
$$
\Ric(N, N)=\sum_i K\left(e_i, N\right)=-\sum_i \frac{h_i^{\prime \prime}}{h_i},
$$
and
$$
\begin{aligned}
	\operatorname{Ric}\left(e_j, e_j\right) & =K\left(e_j, N\right)+\sum_{i \neq j} K\left(e_j, e_i\right) \\
	& =-\frac{h_j^{\prime \prime}}{h_j}-\left(\sum_{i \neq j} \frac{h_i^{\prime}}{h_i}\right) \frac{h_j^{\prime}}{h_j}=-\frac{h_j^{\prime \prime}}{h_j}-\left(\sum_i \frac{h_i^{\prime}}{h_i}-\frac{h_j^{\prime}}{h_j}\right) \frac{h_j^{\prime}}{h_j} \\
	& =-\left(\sum_i \frac{h_i^{\prime}}{h_i}\right) \frac{h_j^{\prime}}{h_j}-\left(\frac{h_j^{\prime \prime}}{h_j}-\left(\frac{h_j^{\prime}}{h_j}\right)^2\right).
\end{aligned}
$$
From these results, we imply that the Scalar curvature is given by 
$$
\begin{aligned}
	S & =\Ric(N, N)+\sum_j \Ric\left(e_j, e_j\right) \\
	& =-\sum_i \frac{h_i^{\prime \prime}}{h_i}-\left(\sum_i \frac{h_i^{\prime}}{h_i}\right)\left(\sum_j \frac{h_j^{\prime}}{h_j}\right)-\sum_j\left(\frac{h_j^{\prime \prime}}{h_j}-\left(\frac{h_j^{\prime}}{h_j}\right)^2\right) \\
	&=-2 \sum_i \frac{h_i^{\prime \prime}}{h_i}-A^2+B,
\end{aligned}
$$
where $A:=\sum_i \frac{h_i^{\prime}}{h_i}, B:=\sum_i\left(\frac{h_i^{\prime}}{h_i}\right)^2.$ Thus, generically, the Weyl tensor is NOT vanishing. 

Plugging the above results in \eqref{e2.4}, we conclude that
$$
\begin{aligned}
	\lambda & =-\sum_i\left(\frac{h_i^{\prime \prime}}{h_i}-\left(\frac{h_i^{\prime}}{h_i}\right)^2\right)-\sum_i\left(\frac{h_i^{\prime}}{h_i}\right)^2+f^{\prime \prime}\\
	& =\left(f^{\prime}-\sum_i \frac{h_i^{\prime}}{h_i}\right) \frac{h_j^{\prime}}{h_j}-\left(\frac{h_j^{\prime \prime}}{h_j}-\left(\frac{h_j^{\prime}}{h_j}\right)^2\right) .
\end{aligned}
$$
Let $u_0:= f'$ and $u_i:=\frac{h_i'}{h_i}$, the above system can be written as follows
\begin{align}\label{e4.3}
\left\{\begin{array}{l}
	A=\Sigma_i u_i \\
	B=\Sigma_i u_i^2 \\
	u_j^{\prime}=\left(u_0-A\right) u_j-\lambda \\
	u_0^{\prime}=B+\left(u_0-A\right) A-(n-1) \lambda.
\end{array}\right.
\end{align}
This is a system of first order ODEs and the Picard-Lindel\"{o}f theorem yields local existence and uniqueness. 

		Globally, we suspect there is no complete solution as $u_0$ might blow up in a finite time. 
\begin{proposition}\label{p4.2}
	Let $(M^n, g, f)\, \left(n\geq 3\right)$ be a steady gradient Ricci soliton with local Euclidean level sets \eqref{e4.2}. Then the metric $g$ is incomplete. 
\end{proposition}
\begin{proof}[Proof of Proposition \ref{p4.2}]
For $\lambda=0$, from the system \eqref{e4.3}, we get 
\begin{align}\label{q4.4}
\left\{\begin{array}{l}
	u_j^{\prime}=\left(u_0-A\right) u_j \\
	u_0^{\prime}=B+\left(u_0-A\right) A.
\end{array}\right.
\end{align}
Note that 
$$
A^{\prime}=\sum_i u_i^{\prime}=\sum_i\left(u_0-A\right) u_i=\left(u_0-A\right) \sum_i u_i=\left(u_0-A\right) A.
$$
We rewrite the system \eqref{q4.4} as
\begin{align}\label{e4.5}
\left\{\begin{aligned}
	u_j^{\prime} & =\left(u_0-A\right) u_j \\
	\left(u_0-A\right)^{\prime} & =B.
\end{aligned}\right.
\end{align}
By the first equation of the above system, we obtain
$
\frac{u_j^{\prime}}{u_j}=\frac{u_i^{\prime}}{u_i}
$
for all $j, i$. This implies that there is a smooth function $h$ such that 
$
\frac{u_j^{\prime}}{u_j}=\frac{h^{\prime}}{h}
$ for all $j.$ Then, we have
$$
\left(\frac{u_j}{h}\right)^{\prime}=\frac{u_j^{\prime} h-u_j h^{\prime}}{h^2}=0 .
$$
Thus $u_j=a_j h$ for some constant $a_j.$ From this and \eqref{e4.5}, one finds that
\begin{align}\label{e4.6}
\left\{\begin{aligned}
	h^{\prime} & =l h \\
l^{\prime} & =b h^2,
\end{aligned}\right. 
\end{align}
where $l=u_0-ah, a=\sum_i a_i, b=\sum_i a^2_i.$ One can notice that
$$
\frac{{\rm d}l}{{\rm d} h}=\frac{{\rm d} l}{{\rm d} t} \frac{{\rm d} t}{{\rm d} h}=\frac{b}{l} h.
$$
This implies that 
$$
\int l {\rm d} l=\int b h {\rm d} h.
$$
Consequently,
$$
b h^2=l^2+C
$$
for some constant $C.$ This and \eqref{e4.6} lead to
\begin{align}\label{4.7}
l^{\prime}=l^2+C.
\end{align}

Now, we consider three possible cases.\\
\textbf{Case 1}: $C=0.$ Then the equation \eqref{4.7} becomes $l'=l^2.$ Using this, we find that 
$$
l(t)=-\frac{1}{t+C_1}, \quad h(t)= \pm \frac{1}{\sqrt{b}\left(t+C_1\right)},
$$
and 
$$
u_0(t)=l(t)+a h(t)=-\frac{1}{t+C_1} \pm \frac{a}{\sqrt{b}\left(t+C_1\right)},
$$
for some constant $C_1.$\\
\textbf{Case 2}: $C>0.$ Then we set $C=D^2$ for some constant $D$ and the equation \eqref{4.7} becomes $l'=l^2+D^2.$ Thus, we have 
$$
l(t)=D \tan \left(D t+D_1\right), \quad h(t)= \pm \frac{D}{\sqrt{b} \cos \left(D t+D_1\right)},
$$
and 
$$
u_0(t)=l(t)+a h(t)=D \tan \left(D t+D_1\right) \pm \frac{a D}{\sqrt{b} \cos \left(D t+D_1\right)},
$$
for some constant $D_1.$\\
\textbf{Case 3}: $C<0.$ Then we set $C=-D^2$ for some constant $D$ and the equation \eqref{4.7} becomes $l'=l^2-D^2.$ From this, we get
$$
l(t)=\frac{D\left(e^{2 D t}+D_1\right)}{e^{2 D t}-D_1}, \quad h(t)= \pm \frac{4 D_1 D^2 e^{2 D t}}{\sqrt{b}\left(e^{2 D t}-D_1\right)},
$$
and 
$$
u_0(t)=l(t)+a h(t)=\frac{D\left(e^{2 D t}+D_1\right)}{e^{2 D t}-D_1} \pm \frac{4 a D_1 D^2 e^{2 D t}}{\sqrt{b}\left(e^{2 D t}-D_1\right)}
$$
for some constant $D_1>0.$

From the above results, we see that the function $u_0$ blows up as $t$ approaches a finite time. Then the metric $g$ is incomplete. 
\end{proof}
\section*{Acknowledgment } 

The first author was supported by Vietnam National Foundation for Science and Technology Development (NAFOSTED) under grant number 101.02-2021.28. The second author's research is partially supported by grants from the National Science Foundation [DMS-2104988] and the Vietnam Institute for Advanced Study in Mathematics.

\fontsize{12}{12}\selectfont
	
	\bigskip
\address{{\it Ha Tuan Dung}\\
Faculty of Mathematics \\
Hanoi Pedagogical University 2  \\
Xuan Hoa, Vinh Phuc, Vietnam 
}
{hatuandung@hpu2.edu.vn}
\address{ {\it Hung Tran}\\
	Department of Mathematics\\
	and Statistics \\
Texas Tech University, Lubbock, TX 79413, USA 
}
{hung.tran@ttu.edu}

%

\end{document}